\documentclass[12pt]{article}
\usepackage[centertags]{amsmath}
\usepackage{amsfonts}
\usepackage{amssymb}
\usepackage{latexsym}
\usepackage{amsthm}
\usepackage{newlfont}
\usepackage{graphicx}
\usepackage{listings}
\usepackage{booktabs}
\usepackage{abstract}
\date{26 May 2017}

\newlength{\defbaselineskip}
\newcommand{\setlinespacing}[1]%
           {\setlength{\baselineskip}{#1 \defbaselineskip}}

\newcommand{\bR}{{\mathbb{R}}}

\newcommand{\N}{{\mathbb{N}}}

\newcommand{\actaqed}{\hfill $\actabox$}
{\medskip\noindent \textit{Proof of #1. }}%
{\actaqed \medskip}

\def\D{{\mathcal D}}

\def\C{{\mathcal C}}

\def \Tr{\mathcal T}

\def \V{\mathcal V}

\def \cM{\mathcal M}
\def\R{{\mathbb R}}
\def\Z{\mathbb Z}

\def \T{\mathbb T}
\def\bP{\mathbb P}
\def\bE{\mathbb E}
\def\bbC{\mathbb C}
\def \<{\langle}
\def\>{\rangle}
\def \L{\Lambda}

\def \e{\varepsilon}
\def \de{\delta}

\def \ff{\varphi}
\def\la{\lambda}

\def \sp{\operatorname{span}}

\def \sign{\operatorname{sign}}

\def\bx{\mathbf x}

\def\bz{\mathbf z}
\def\bk{\mathbf k}
\def\bv{\mathbf v}
\def\bw{\mathbf w}
\def\bm{\mathbf m}
\def\bs{\mathbf s}
\def\bN{\mathbf N}

\newtheorem{Theorem}{Theorem}[section]
\newtheorem{Lemma}{Lemma}[section]

\newtheorem{Proposition}{Proposition}[section]
\newtheorem{Remark}{Remark}[section]

\numberwithin{equation}{section}

\newcommand{\be}{\begin{equation}}
\newcommand{\ee}{\end{equation}}

\begin{document}

\title{The Marcinkiewicz-type discretization theorems for the hyperbolic cross polynomials}
\author{V.N. Temlyakov\thanks{University of South Carolina and Steklov Institute of Mathematics.  }}
\maketitle
\begin{abstract}
{The main goal of this paper is to study the discretization problem for the hyperbolic cross trigonometric polynomials. This important problem turns out to be very difficult. In this paper we begin a systematic study of this problem and demonstrate two different techniques -- the probabilistic and the number theoretical techniques.}
\end{abstract}

\section{Introduction} 
\label{I} 

Discretization is 
a very important step in making a continuous problem computationally feasible. The problem of construction of good sets of points in a multidimensional domain is a fundamental problem of mathematics and computational mathematics.
We note that the problem of arranging points in a multidimensional domain is also a fundamental problem in coding theory. It is a problem on optimal spherical codes. This problem is equivalent to the problem from compressed sensing on building 
large incoherent dictionaries in $\R^d$.   

A prominent  example of classical discretization problem is a problem of metric entropy (covering numbers, entropy numbers). 
Bounds for the $\varepsilon$-entropy of function classes are important by themselves and also have important connections to other fundamental problems. We give one new upper bound for the entropy numbers in Section \ref{entropy}.
 
Another prominent example of a discretization problem is the problem of numerical integration. It turns out that contrary to the numerical integration in the univariate case and in the multivariate case of the isotropic and anisotropic Sobolev and Nikol'skii smoothness classes (see \cite{TBook}, Ch.2), where regular grids methods are optimal (in the sense of order), 
in the case of numerical integration of functions with mixed smoothness the regular grids methods are very far from being optimal. Numerical integration in the mixed smoothness classes requires deep number theoretical results for constructing optimal (in the sense of order) cubature formulas.  

A problem of optimal recovery is one more example of a discretization problem. This problem turns out to be very difficult for the mixed smoothness classes. It is not solved even in the case of optimal recovery in the $L_2$ norm. It is well known (see, for instance, \cite{DTU}) that the hyperbolic cross polynomials play a fundamental role in approximation of functions from mixed smoothness classes. The main goal of this paper is to study the discretization problem for the hyperbolic cross trigonometric polynomials. This important problem turns out to be very difficult. In this paper we begin a systematic study of this problem and demonstrate two different techniques -- the probabilistic and the number theoretical techniques. 

Let $\Omega$ be a compact subset of $\R^d$ with the probability measure $\mu$. We say that a linear subspace $X_N$ (index $N$ here, usually, stands for the dimension of $X_N$)  of the $L_q(\Omega)$, $1\le q < \infty$, admits the Marcinkiewicz-type discretization theorem with parameters $m$ and $q$ if there exist a set $\{\xi^\nu \in \Omega, \nu=1,\dots,m\}$ and two positive constants $C_j(d,q)$, $j=1,2$, such that for any $f\in X_N$ we have
\be\label{1.1}
C_1(d,q)\|f\|_q^q \le \frac{1}{m} \sum_{\nu=1}^m |f(\xi^\nu)|^q \le C_2(d,q)\|f\|_q^q.
\ee
In the case $q=\infty$ we define $L_\infty$ as the space of continuous on $\Omega$ functions and ask for 
\be\label{1.2b}
C_1(d)\|f\|_\infty \le \max_{1\le\nu\le m} |f(\xi^\nu)| \le  \|f\|_\infty.
\ee

We will also use a brief way to express the above property: the $\cM(m,q)$ theorem holds for  a subspace $X_N$ or $X_N \in \cM(m,q)$. 
Relation (\ref{1.1}) allows us to discretize the $L_q$ norm of any element in $X_N$ with respect to the net $\{\xi^\nu\}_{\nu=1}^m$. 
Here we concentrate on the periodic case of $d$-variate functions. Thus in our case $\Omega =\T^d$, $\mu$ is a normalized Lebesgue measure on $\T^d$. 
We study the  Marcinkiewicz-type discretization theorems for subspaces of the 
trigonometric polynomials. Let $Q$ be a finite subset of $\Z^d$. We denote
$$
\Tr(Q):= \{f: f=\sum_{\bk\in Q}c_\bk e^{i(\bk,\bx)}\}.
$$
We briefly present well known results related to the Marcinkiewicz-type discretization theorems for the trigonometric polynomials. We begin with the case $Q=\Pi(\bN):=[-N_1,N_1]\times \cdots \times [-N_d,N_d]$, $N_j \in \N$ or $N_j=0$, $j=1,\dots,d$, $\bN=(N_1,\dots,N_d)$. 
We denote
\begin{align*}
P(\mathbf N) := \bigl\{\mathbf n = (n_1 ,\dots,n_d),&\qquad n_j\ -\
\text{ are nonnegative integers},\\
&0\le n_j\le 2N_j  ,\qquad j = 1,\dots,d \bigr\},
\end{align*}
and set
$$
\bx^{\mathbf n}:=\left(\frac{2\pi n_1}{2N_1+1},\dots,\frac{2\pi n_d}
{2N_d+1}\right),\qquad \mathbf n\in P(\mathbf N) .
$$
Then for any $t\in \Tr(\Pi(\mathbf N))$
\be\label{b1.5}
\|t\|_2^2  =\vartheta(\mathbf N)^{-1}\sum_{\mathbf n\in P(\mathbf N)}
\bigl|t(\bx^{\mathbf n})\bigr|^2   
\ee
where $\vartheta(\mathbf N) := \prod_{j=1}^d (2N_j  + 1)=\dim\Tr(\Pi(\bN))$.  

In particular, relation (\ref{b1.5}) shows that for any $\bN$ we have
\be\label{1.3}
\Tr(\Pi(\bN)) \in \cM(\vartheta(\bN),2).
\ee
The following version of (\ref{b1.5}) for $1<q<\infty$ is the well known Marcinkiewicz discretization theorem (for $d=1$) (see \cite{Z}, Ch.10, \S7 and \cite{TBook}, Ch.1, Section 2)
$$
C_1(d,q)\|t\|_q^q  \le\vartheta(\mathbf N)^{-1}\sum_{\mathbf n\in P(\mathbf N)}
\bigl|t(\bx^{\mathbf n})\bigr|^q  \le C_2(d,q)\|t\|_q^q,\quad 1<q<\infty,
$$
which implies the following generalization of (\ref{1.3})
\be\label{1.4}
\Tr(\Pi(\bN)) \in \cM(\vartheta(\bN),q),\quad 1<q<\infty.
\ee
Some modifications are needed in the case $q=1$ or $q=\infty$. 
Denote
\begin{align*}
P'(\mathbf N) := \bigl\{\mathbf n &= (n_1,\dots,n_d),\qquad n_j\ -\
\text{ are natural numbers},\\
&1\le n_j\le 4N_j  ,\qquad j = 1,\dots,d \bigr\}
\end{align*}
and set
$$
\bx(\mathbf n) :=\left (\frac{\pi n_1}{2N_1} ,\dots,\frac{\pi n_d}{2N_d}
\right)   ,\qquad \mathbf n\in P'(\mathbf N)  .
$$
In the case $N_j  = 0$ we assume $x_j (\mathbf n) = 0$. Denote ${\overline N} := \max (N,1)$ and $\nu(\bN) := \prod_{j=1}^d {\overline N_j}$. Then the following Marcinkiewicz-type discretization theorem is known
$$
C_1(d,q)\|t\|_q^q  \le\nu(4\mathbf N)^{-1}\sum_{\mathbf n\in P'(\mathbf N)}
\bigl|t(\bx({\mathbf n}))\bigr|^q  \le C_2(d,q)\|t\|_q^q,\quad 1\le q\le \infty,
$$
which implies the following relation
\be\label{1.5}
\Tr(\Pi(\bN)) \in \cM(\nu(4\bN),q),\quad 1\le q\le \infty.
\ee
Note that $\nu(4\bN) \le C(d) \dim \Tr(\Pi(\bN))$. 

In this paper we are primarily interested in the Marcinkiewicz-type discretization theorems for the hyperbolic cross trigonometric polynomials. For $\bs\in\Z^d_+$
define
$$
\rho (\bs) := \{\bk \in \Z^d : [2^{s_j-1}] \le |k_j| < 2^{s_j}, \quad j=1,\dots,d\}
$$
where $[x]$ denotes the integer part of $x$. We define the step hyperbolic cross
$Q_n$ as follows
$$
Q_n := \cup_{\bs:\|\bs\|_1\le n} \rho(\bs)
$$
and the corresponding set of the hyperbolic cross polynomials as 
$$
\Tr(Q_n) := \{f: f=\sum_{\bk\in Q_n} c_\bk e^{i(\bk,\bx)}\}.
$$
The problem on the right  Marcinkiewicz-type discretization theorems for the hyperbolic cross trigonometric polynomials is wide open. There is no sharp results on the growth of $m$ as a function on $n$ for the relation $\Tr(Q_n) \in \cM(m,q)$ to hold  for  $1\le q\le \infty$, $q\neq 2$. Clearly, $Q_n \subset \Pi(2^n,\dots,2^n)$, and therefore the above discussed results give
$$
\Tr(Q_n)\in \cM(m,q),\quad \text{provided} \quad m\ge C(d)2^{dn},\quad 1\le q\le \infty,
$$
with large enough $C(d)$.
Probably, the first  nontrivial result  in this direction was obtained in \cite{VT27}, where the set of points $\{\xi^\nu\}_{\nu=1}^p$ with $p\ll 2^{2n}n^{d-1}$ such that for all $t\in \Tr(Q_n)$ inequality
$$
\|t\|_2^2 \le \frac{1}{p} \sum_{\nu=1}^p |t(\xi^\nu)|^2
$$
holds was constructed. Later a very nontrivial surprising negative result was proved for $q=\infty$ (see \cite{KT3}, \cite{KT4}, and \cite{KaTe03}). The authors proved that the necessary condition for
$\Tr(Q_n)\in\cM(m,\infty)$ is $m\gg |Q_n|^{1+c}$ with absolute constant $c>0$.
There are deep general results about submatrices of orthogonal matrices, which 
provide very good Marcinkiewicz-type discretization theorems for $q=2$. For example, the Theorem from \cite{Rud} gives the following result
\be\label{1.5a}
\Tr(Q_n)\in \cM(m,2),\quad \text{provided} \quad m\ge C(d)|Q_n|n 
\ee
with large enough $C(d)$. 

We now comment on a recent breakthrough result by J. Batson, D.A. Spielman, and N. Srivastava \cite{BSS}. We formulate their result in our notations. Let  $\Omega_M=\{x^j\}_{j=1}^M$ be a discrete set with the probability measure $\mu(x^j)=1/M$, $j=1,\dots,M$. Assume that 
$\{u_i(x)\}_{i=1}^N$ is a real orthonormal on $\Omega_M$ system. Then for any 
number $d>1$ there exist a set of weights $w_j\ge 0$ such that $|\{j: w_j\neq 0\}| \le dN$ so that for any $f\in X_N:= \sp\{u_1,\dots,u_N\}$ we have (see \cite{VT161})
$$
\|f\|_2^2 \le \sum_{j=1}^M w_jf(x^j)^2 \le \frac{d+1+2\sqrt{d}}{d+1-2\sqrt{d}}\|f\|_2^2.
$$
In particular, this implies that the $L_2$ Marcinkiewicz-type discretization theorem holds for the above $X_N$ with $m\ge cN$ in a modified form -- we allow general weights $w_j$ instead of weights $1/m$ in formula (\ref{1.1}). 

In Section \ref{L2} we show how to derive the following result from the recent paper by  S. Nitzan, A. Olevskii, and A. Ulanovskii \cite{NOU}, which in turn is based on the paper of A. Marcus, D.A. Spielman, and N. Srivastava \cite{MSS}. 

\begin{Theorem}\label{NOUth} There are three positive absolute constants $C_1$, $C_2$, and $C_3$ with the following properties: For any $d\in \N$ and any $Q\subset \Z^d$   there exists a set of  $m \le C_1|Q| $ points $\xi^j\in \T^d$, $j=1,\dots,m$ such that for any $f\in \Tr(Q)$ 
we have
$$
C_2\|f\|_2^2 \le \frac{1}{m}\sum_{j=1}^m |f(\xi^j)|^2 \le C_3\|f\|_2^2.
$$
\end{Theorem}

 Theorem \ref{NOUth}, basically, solves the problem of the Marcinkiewicz-type discretization theorem for the $\Tr(Q_n)$ in the $L_2$ case.
The reader can find some more discussion of the $L_2$ case in 
\cite{VT161}. We also refer to the paper \cite{Ka} for a discussion of a
recent outstanding progress in the area of submatrices of orthogonal matrices. 

The most important results of this paper are in Section \ref{L1}. We prove there that for $d=2$ (see Theorem \ref{T2.1})
$$
\Tr(Q_n) \in \cM(m,1),\quad \text{provided} \quad m\ge C |Q_n|n^{7/2}
$$
with large enough $C$,
and for $d\ge 3$ (see Theorem \ref{T2.3})
$$
\Tr(Q_n) \in \cM(m,1),\quad \text{provided} \quad m\ge C(d) |Q_n|n^{d/2+3}
$$
with large enough $C(d)$.
A very interesting open problem is the following. Does the relation $\Tr(Q_n) \in \cM(m,1)$ hold with $ m\asymp |Q_n|$? 

The above results of Section \ref{L1} are obtained with a help of probabilistic technique. We use a variant of the Bernstein concentration measure inequality 
from \cite{BLM}, the chaining technique from \cite{KoTe} (see also \cite{Tbook}, Ch.4), and the bounds of the entropy numbers from a very recent paper \cite{VT156}. We note that the idea of chaining technique goes back to the 1930s, when it was suggested by A.N. Kolmogorov. Later, these type of results have been developed in the study of the central limit theorem in probability theory (see, for instance, \cite{GZ}). The reader can find further results on the chaining technique in \cite{Ta}. 

In Section \ref{Lq} we extend the technique developed in Section \ref{L1} to the case $1<q< \infty$. In this case we are only able to prove the following relation (see Theorem \ref{T3.3})
\be\label{1.6}
\Tr(Q_n) \in \cM(m,q),\quad \text{provided} \quad m\ge C(d,q) |Q_n|^{2-1/q}n^{a(d,q)}
\ee
with some large enough constants $C(d,q)$ and $a(d,q)$. A very interesting open problem here is the following. Does the relation $\Tr(Q_n) \in \cM(m,q)$ hold for $ m\ge C(d,q) |Q_n|n^{c(d,q)}$ with some constants $C(d,q)$ and $c(d,q)$? As we pointed out above, in the case $q=2$ the answer to this question is "yes". Moreover, we can take $c(d,q)=0$ (see Theorem \ref{NOUth}). This indicates that our technique from Section \ref{L1}, which works reasonably well for $q=1$, is not a good technique for $q>1$. 

As we already pointed out above the technique developed in Sections \ref{L1} and \ref{Lq} is the probabilistic technique. This allows us to prove existence of 
good points for the Marcinkiewicz-type discretization theorems but it does not provide an algorithm of construction of these points. It would be very interesting 
to provide deterministic constructions of points sets, which give at least the same bounds for $m$ as the probabilistic technique does. In Section \ref{L2} we 
present a deterministic construction, which is based on number theoretical 
considerations. This technique works for any finite set $Q\subset \Z^d$. However, it is limited to the case $q=2$. Theorem \ref{NOUth}   shows that for $q=2$ the probabilistic technique provides the Marcinkiewicz-type discretization theorem for $m\ge C |Q_n| $ with some large enough constant $C$. In Section \ref{L2} we only prove the Marcinkiewicz-type discretization theorem for $m\ge C(d) |Q_n|^2 $ with large enough constant $C(d)$. However, we prove the Marcinkiewicz-type discretization theorem in the strong form with $C_1(d,2)=C_2(d,2)=1$. Namely, for a given set $Q$ we construct a set $\{\xi^\nu\}_{\nu=1}^m$ with $m\le C(d)|Q|^2$ such that for any $t\in \Tr(Q)$ we have
$$
\|t\|_2^2 = \frac{1}{m}\sum_{\nu=1}^m |t(\xi^\nu)|^2.
$$

The probabilistic technique developed in Sections \ref{L1} and \ref{Lq} requires 
bounds on the entropy numbers $\e_k(\Tr(Q_n)_q,L_\infty)$ of the unit $L_q$ balls of $\Tr(Q_n)$ in $L_\infty$. This problem by itself is a deep and difficult problem. Recently, a new method based on greedy approximation approach was developed in \cite{VT156}. We use results from \cite{VT156} in Sections \ref{L1} and \ref{Lq}. Section \ref{entropy} complements \cite{VT156} by results on the upper bounds for the $\e_k(\Tr(Q_n)_1,L_\infty)$, $d\ge 3$.

\section{The Marcinkiewicz-type theorem in $L_1$. Probabilistic technique}
\label{L1}

We begin with two lemmas, which are analogs of the well known concentration measure inequalities (see, for instance \cite{Tbook}, Ch.4). Lemma \ref{L2.1} is from \cite{BLM}.
\begin{Lemma}\label{L2.1} Let $\{g_j\}_{j=1}^m$ be independent random variables with $\bE g_j=0$, $j=1,\dots,m$, which satisfy
$$
\|g_j\|_1\le 2,\qquad \|g_j\|_\infty \le M,\qquad j=1,\dots,m.
$$
Then for any $\eta \in (0,1)$ we have the following bound on the probability
$$
\bP\left\{\left|\sum_{j=1}^m g_j\right|\ge m\eta\right\} < 2\exp\left(-\frac{m\eta^2}{8M}\right).
$$
\end{Lemma}
\begin{Lemma}\label{L2.2} Let $\{g_j\}_{j=1}^m$ be independent random variables with $\bE g_j=0$, $j=1,\dots,m$, which satisfy
$$
\|g_j\|_2\le 2,\qquad \|g_j\|_\infty \le M,\qquad j=1,\dots,m.
$$
Then   we have the following bound on the probability
$$
\bP\left\{\left|\sum_{j=1}^m g_j\right|\ge m\eta\right\} < 2\left\{\begin{array}{ll}\exp\left(-\frac{m\eta^2}{8}\right),&\quad \eta \le 4/M,\\ \exp\left(-\frac{m\eta}{2M}\right),&\quad \eta > 4/M.\end{array}\right.
$$
\end{Lemma}
\begin{proof} The proofs of both lemmas are similar. For completeness we give 
the detailed proof of Lemma \ref{L2.1} from \cite{BLM} and give a comment on the modifications of this proof, which prove Lemma \ref{L2.2}. We use the inequality $e^x \le 1+x+x^2$, $x\le 1$. Then for $0<\lambda M\le 1$ we have
$$
\int \exp(\lambda g_j) d\mu \le 1+\lambda^2\int g_j^2 d\mu.
$$
In the proof of Lemma \ref{L2.1} we bound
$$
1+\lambda^2\int g_j^2 d\mu \le 1+\lambda^2\|g_j\|_1\|g_j\|_\infty \le \exp(2\lambda^2M).
$$
In the proof of Lemma \ref{L2.2} we bound
$$
1+\lambda^2\int g_j^2 d\mu \le 1+\lambda^2\|g_j\|_2 \le \exp(2\lambda^2).
$$
Therefore, in the proof of Lemma \ref{L2.1} we get
\be\label{2.1}
\int\exp\left(\lambda\sum_{j=1}^m g_j\right) d\mu \le \exp(2m\lambda^2M)
\ee
and in the proof of Lemma \ref{L2.2} we obtain
$$
\int\exp\left(\lambda\sum_{j=1}^m g_j \right)d\mu \le \exp(2m\lambda^2).
$$
We now demonstrate how to complete the proof of Lemma \ref{L2.1}. The completion of proof of Lemma \ref{L2.2} goes along the same lines. Inequality (\ref{2.1}) implies
$$
\exp(\lambda m\eta)\bP\left\{\sum_{j=1}^m g_j\ge m\eta\right\}\le\int\exp\left(\lambda\sum_{j=1}^m g_j \right)d\mu \le \exp(2m\lambda^2M).
$$
Choosing $\lambda =\eta/(4M)$ we complete the proof.
\end{proof}

We now consider measurable functions $f(\bx)$, $\bx\in \Omega$. For $1\le q<\infty$ define
$$
L^q_\bz(f) := \frac{1}{m}\sum_{j=1}^m |f(\bx^j)|^q -\|f\|_q^q,\qquad \bz:= (\bx^1,\dots,\bx^m).
$$
Let $\mu$ be a probabilistic measure on $\Omega$. Denote $\mu^m := \mu\times\cdots\times\mu$ the probabilistic measure on $\Omega^m := \Omega\times\cdots\times\Omega$.
We will need the following inequality, which is a corollary of Lemma \ref{L2.1}. 
\begin{Proposition}\label{P2.1} Let $f_j\in L_1(\Omega)$ be such that
$$
\|f_j\|_1 \le 1/2,\quad j=1,2;\qquad \|f_1-f_2\|_\infty \le \delta.
$$
Then
\be\label{2.2}
\mu^m\{\bz: |L^1_\bz(f_1) -L^1_\bz(f_2)| \ge \eta\} < 2\exp\left(-\frac{m\eta^2}{16\delta}\right).
\ee
\end{Proposition}
\begin{proof} Consider the function
$$
g(\bx) := |f_1(\bx)|-\|f_1\|_1 - (|f_2(\bx)|-\|f_2\|_1).
$$
Then $\int g(\bx)d\mu =0$ and
$$
\|g\|_1 \le 2\|f_1\|_1 +2\|f_2\|_1 \le 2,\qquad \|g\|_\infty \le 2\|f_1-f_2\|_\infty \le 2\delta.
$$
Consider $m$ independent variables $\bx^j\in \Omega$, $j=1,\dots,m$. For $\bz\in \Omega^m$ define $m$ independent random variables $g_j(\bz)$ as $g_j(\bz) := g(\bx^j)$. Clearly,
$$
\frac{1}{m}\sum_{j=1}^m g_j = L^1_\bz(f_1)-L^1_\bz(f_2).
$$
Applying Lemma \ref{L2.1} with $M=2\delta$ to $\{g_j\}$ we obtain (\ref{2.2}).
\end{proof}

We now prove the Marcinkiewicz-type theorem for discretization of the $L_1$ norm of the bivariate hyperbolic cross polynomials. 
\begin{Theorem}\label{T2.1} For any $n\in \N$ there exists a set of $m \le C_1|Q_n|n^{7/2}$ points $\xi^j\in \T^2$, $j=1,\dots,m$ such that for any $f\in \Tr(Q_n)$ 
we have
$$
C_2\|f\|_1 \le \frac{1}{m}\sum_{j=1}^m |f(\xi^j)| \le C_3\|f\|_1.
$$
\end{Theorem}
\begin{proof} Proposition \ref{P2.1} plays an important role in our proof. It is used in the proof of the bound on the probability of the event 
$\{\sup_{f\in W}|L^1_\bz(f)|\ge \eta\}$ for a function class $W$. The corresponding proof is
  in terms of the entropy numbers of $W$. We now introduce the corresponding definitions. 
  
  Let $X$ be a Banach space and let $B_X$ denote the unit ball of $X$ with the center at $0$. Denote by $B_X(y,r)$ a ball with center $y$ and radius $r$: $\{x\in X:\|x-y\|\le r\}$. For a compact set $A$ and a positive number $\e$ we define the covering number $N_\e(A)$
 as follows
$$
N_\e(A) := N_\e(A,X) 
:=\min \{n : \exists y^1,\dots,y^n, y^j\in A :A\subseteq \cup_{j=1}^n B_X(y^j,\e)\}.
$$
It is convenient to consider along with the entropy $H_\e(A,X):= \log_2 N_\e(A,X)$ the entropy numbers $\e_k(A,X)$:
$$
\e_k(A,X)  :=\inf \{\e : \exists y^1,\dots ,y^{2^k} \in A : A \subseteq \cup_{j=1}
^{2^k} B_X(y^j,\e)\}.
$$
In our definition of $N_\e(A)$ and $\e_k(A,X)$ we require $y^j\in A$. In a standard definition of $N_\e(A)$ and $\e_k(A,X)$ this restriction is not imposed. 
However, it is well known (see \cite{Tbook}, p.208) that these characteristics may differ at most by a factor $2$. 

We consider the case $X$ is 
 $\C(\Omega)$ the space of functions continuous on a compact subset $\Omega$ of $\bR^d$ with the norm
$$
\|f\|_\infty:= \sup_{\bx\in \Omega}|f(\bx)|.
$$
  We use the abbreviated notation
$$
  \e_n(W):= \e_n(W,\C).
$$
In our case $d=2$ and
\be\label{2.3}
W:=W(n) := \{t\in \Tr(Q_n): \|t\|_1 = 1/2\}.
\ee
The following result on the entropy numbers for the $W$ is from \cite{VT156}.
Denote
$$
\Tr(Q)_q := \{t\in \Tr(Q): \|t\|_q\le 1\}.
$$
\begin{Theorem}\label{T2.2} We have for $d=2$
$$
\e_k(\Tr( Q_n)_1,L_\infty)\le 2\e_k := 2C_4  \left\{\begin{array}{ll} n^{1/2}(| Q_n|/k) \log (4| Q_n|/k), &\quad k\le 2| Q_n|,\\
n^{1/2}2^{-k/(2| Q_n|)},&\quad k\ge 2| Q_n|.\end{array} \right.
$$
\end{Theorem}
We note that by the following known (see \cite{Tmon}) Nikol'skii-type inequality for the hyperbolic cross polynomials: 
\be\label{2.9b}
\|f\|_\infty \le C(d)2^n \|f\|_1,\qquad f\in \Tr(Q_n)
\ee
we get a bound $\e_k(\Tr( Q_n)_1,L_\infty\le C(d)2^n$ for all $k$. For small $k$ this bound is better than the one provided by Theorem \ref{T2.2}. However, this  improvement of Theorem \ref{T2.2} for small $k$ does not affect our bounds. For convenience we use Theorem \ref{T2.2} in the above form. 

Specify $\eta=1/4$.
Denote $\de_j := \e_{2^j}$, $j=0,1,\dots$, and consider minimal $\de_j$-nets ${\mathcal N}_j \subset W$ of $W$ in $\C(\T^2)$. We use the notation $N_j:= |{\mathcal N}_j|$. Let $J$ be the minimal $j$ satisfying $\de_j \le 1/16$. For $j=1,\dots,J$ we define a mapping $A_j$ that associates with a function $f\in W$ a function $A_j(f) \in {\mathcal N}_j$ closest to $f$ in the $\C$ norm. Then, clearly,
$$
\|f-A_j(f)\|_\C \le \de_j.  
$$
We use the mappings $A_j$, $j=1,\dots, J$ to associate with a function $f\in W$ a sequence (a chain) of functions $f_J, f_{J-1},\dots, f_1$ in the following way
$$
f_J := A_J(f),\quad f_j:=A_j(f_{j+1}),\quad j=1,\dots,J-1.
$$
Let us find an upper bound for $J$, defined above. Certainly, we can carry out the proof under assumption that $C_4\ge 1$. Then the definition of $J$ implies that $2^J\ge 2|Q_n|$ and
\be\label{2.4}
C_4n^{1/2}2^{-2^{J-1}/(2| Q_n|)} \ge 1/16.
\ee
We derive from (\ref{2.4})
\be\label{2.5}
2^J \le 4|Q_n|+ C\log n,\qquad J \le n+C\log n \le 2n
\ee
for sufficiently large $n\ge C$. 

Set 
$$
\eta_j := \frac{1}{16n},\quad j=1,\dots,J.
$$

We now proceed to the estimate of $\mu^m\{\bz:\sup_{f\in W}|L^1_\bz(f)|\ge 1/4\}$. First of all   by the following simple  inequality (\ref{P2.2})   the assumption $\de_J\le 1/16$ implies that if $|L^1_\bz(f)| \ge 1/4$ then $|L^1_\bz(f_J)|\ge 1/8$.  
\be\label{P2.2}  
|L^1_\bz(f_1)-L^1_\bz(f_2)| \le 2\delta\quad\text{provided}\quad \|f_1-f_2\|_\infty\le \delta .
\ee
Rewriting 
$$
L^1_\bz(f_J) = L^1_\bz(f_J)-L^1_\bz(f_{J-1}) +\dots+L^1_\bz(f_{I+1})-L^1_\bz(f_1)+L^1_\bz(f_1)
$$
we conclude that if $|L^1_\bz(f)| \ge 1/4$ then at least one of the following events occurs:
$$
|L^1_\bz(f_j)-L^1_\bz(f_{j-1})|\ge \eta_j\quad\text{for some}\quad j\in (1,J] \quad\text{or}\quad |L^1_\bz(f_1)|\ge \eta_1.
$$
Therefore
\begin{eqnarray}\label{2.6}
\mu^m\{\bz:\sup_{f\in W}|L^1_\bz(f)|\ge1/4\}
\le \mu^m\{\bz:\sup_{f\in {\mathcal N}_1}|L^1_\bz(f)|\ge\eta_1\} \nonumber \\
+\sum_{j\in(1,J]}\sum_{f\in {\mathcal N}_j}\mu^m
\{\bz:|L^1_\bz(f)-L^1_\bz(A_{j-1}(f))|\ge\eta_j\}\nonumber\\
\le \mu^m\{\bz:\sup_{f\in {\mathcal N}_1}|L^1_\bz(f)|\ge\eta_1\}\nonumber\\
+\sum_{j\in(1,J]} N_j\sup_{f\in W}\mu^m
\{\bz:|L^1_\bz(f)-L^1_\bz(A_{j-1}(f))|\ge\eta_j\}. 
\end{eqnarray}
 Applying  Proposition \ref{P2.1} we obtain
$$
\sup_{f\in W} \mu^m\{\bz:|L^1_\bz(f)-L^1_\bz(A_{j-1}(f))|\ge \eta_j\} \le 2\exp\left(-\frac{m\eta_j^2}{16\de_{j-1}}\right).
$$

We now make further estimates for a specific $m=C_1|Q_n|n^{7/2}$ with large 
enough $C_1$. For $j$ such that $2^j\le 2|Q_n|$ we obtain from the definition of 
$\delta_j$
$$
\frac{m\eta_j^2}{\delta_{j-1}} \ge \frac{C_1n^{3/2}2^{j-1}}{C_5n^{3/2} } \ge \frac{C_1}{2C_5}2^j.
$$
By our choice of $\delta_j=\e_{2^j}$ we get $N_j\le 2^{2^j} <e^{2^j}$ and, therefore,
\be\label{2.7}
N_j\exp\left(-\frac{m\eta_j^2}{16\de_{j-1}}\right)\le \exp(-2^j)
\ee
for sufficiently large $C_1$.

In the case $2^j\in (2|Q_n|, 2^J]$ we have
$$
\frac{m\eta_j^2}{\delta_{j-1}} \ge \frac{C_1|Q_n|n^{3/2}}{C_6n^{1/2}2^{-2^{j-1}/(2|Q_n|)}} \ge \frac{C_1}{C_6}n|Q_n| \ge 2^{J+1}
$$
 for sufficiently large $C_1$ and
\be\label{2.8}
N_j\exp\left(-\frac{m\eta_j^2}{16\de_{j-1}}\right)\le \exp(-2^j).
\ee

We now estimate $\mu^m\{\bz:\sup_{f\in {\mathcal N}_1}|L^1_\bz(f)|\ge\eta_1\}$.
We use Lemma \ref{L2.1} with $g_j(\bz) = |f(\bx^j)|-\|f\|_1$. To estimate $\|g_j\|_\infty$ it is sufficient to use the following trivial Nikol'skii-type inequality for the 
hyperbolic cross polynomials: 
\be\label{2.9}
\|f\|_\infty \le |Q_n| \|f\|_1,\qquad f\in \Tr(Q_n).
\ee
We note that it is known (see \cite{Tmon}) that inequality (\ref{2.9}) can be improved by replacing $|Q_n|$ by $C(d)2^n$. Then Lemma \ref{L2.1} gives
$$
\mu^m\{\bz:\sup_{f\in {\mathcal N}_1}|L^1_\bz(f)|\ge\eta_1\}\le 2N_1\exp\left(-\frac{m\eta_1^2}{C|Q_n|}\right) \le 1/4
$$
($N_1=4$ here) for sufficiently large $C_1$. Substituting the above estimates into (\ref{2.6}) we obtain
$$
\mu^m\{\bz:\sup_{f\in W}|L^1_\bz(f)|\ge1/4\} <1.
$$
Therefore, there exists $\bz_0=(\xi^1,\dots,\xi^m)$ such that for any $f\in W$ we have 
$$
|L^1_{\bz_0}(f)| \le 1/4.
$$
Taking into account that $\|f\|_1=1/2$ for $f\in W$ we obtain the statement of Theorem \ref{T2.1} with $C_2 =1/2$, $C_3=3/2$.
 
\end{proof}

We presented above a detailed proof of Theorem \ref{T2.1}. This theorem only applies to the case $d=2$. The reason for this limitation is the use of Theorem 
\ref{T2.2}, which is proved in \cite{VT156} only for $d=2$. In Section \ref{entropy}
we prove a weaker version of Theorem \ref{T2.2} that holds for all $d$. Replacing Theorem \ref{T2.2} by Theorem \ref{T6.3} in the proof of Theorem \ref{T2.1} we obtain the following result for all $d$. We point out that for $d=2$ 
Theorem \ref{T2.3} is weaker than Theorem \ref{T2.1}.

\begin{Theorem}\label{T2.3} For any $n\in \N$ there exists a set of $m \le C_1(d)|Q_n|n^{d/2+3}$ points $\xi^j\in \T^d$, $j=1,\dots,m$ such that for any $f\in \Tr(Q_n)$ 
we have
$$
C_2\|f\|_1 \le \frac{1}{m}\sum_{j=1}^m |f(\xi^j)| \le C_3\|f\|_1.
$$
\end{Theorem}

\section{The Marcinkiewicz-type theorem in $L_q$, $1<q<\infty$. Probabilistic technique}
\label{Lq}

In this section we demonstrate how the technique developed in Section \ref{L1}
can be extended to the case $1<q\le 2$. Theorem \ref{T2.3} shows that in the case $q=1$ the probabilistic technique provides existence of sets of points $\xi^j\in \T^d$, $j=1,\dots,m$, with $m \le C_1(d)|Q_n|n^{d/2+3}$, for which the 
Marcinkiewicz-type theorem holds in $\Tr(Q_n)$. Clearly, it must be $m\ge |Q_n| =\dim \Tr(Q_n)$. Thus the upper bound for $m$ from Theorem \ref{T2.3} differs from the trivial lower bound by an extra factor, which is a log-type factor in terms of 
$|Q_n|$. The results for $1<q\le 2$, which we present in this section are not of this form. They only guarantee that $m\le C(d,q)|Q_n|^{2-1/q}n^a$. However, we do not know if it could be improved to $m\le C(d,q)|Q_n| n^a$ for $q\neq 2$. We now proceed 
to the detailed presentation. We need the following version of Proposition \ref{P2.1}.

\begin{Proposition}\label{P3.1} Let $f_j\in L_q(\Omega)$, $1<q<\infty$, be such that
$$
\|f_j\|_q \le 1/2,\quad \|f_j\|_\infty \le M_q, \quad j=1,2;\qquad \|f_1-f_2\|_\infty \le \delta.
$$
Then
\be\label{3.2}
\mu^m\{\bz: |L^q_\bz(f_1) -L^q_\bz(f_2)| \ge \eta\} < 2\exp\left(-\frac{m\eta^2}{C(q)M_q^{q-1}\delta}\right).
\ee
\end{Proposition}
\begin{proof} Consider the function
$$
g(\bx) := |f_1(\bx)|^q-\|f_1\|_q^q - (|f_2(\bx)|^q-\|f_2\|_q^q).
$$
Then $\int g(\bx)d\mu =0$ and
$$
\|g\|_1 \le 2\|f_1\|_q^q +2\|f_2\|_q^q \le 2,
$$
$$
 \|g\|_\infty \le C_1(q)M^{q-1}_q\delta.
$$
Consider $m$ independent variables $\bx^j\in \Omega$, $j=1,\dots,m$. For $\bz\in \Omega^m$ define $m$ independent random variables $g_j(\bz)$ as $g_j(\bz) := g(\bx^j)$. Clearly,
$$
\frac{1}{m}\sum_{j=1}^m g_j = L^q_\bz(f_1)-L^q_\bz(f_2).
$$
Applying Lemma \ref{L2.1} with $M=C_1(q)M^{q-1}_q\delta$ to $\{g_j\}$ we obtain (\ref{3.2}).
\end{proof}

We use the following known result on the Nikol'skii-type inequalities, which gives 
an upper bound on $M_q$ (see \cite{Tmon}). 

\begin{Theorem}\label{T3.1} Suppose that $1\le q<\infty$. Then
$$
\sup_{f\in \Tr(Q_n)} \|f\|_\infty/\|f\|_q \asymp 2^{n/q}n^{(d-1)(1-1/q)}.
$$
\end{Theorem}

Using the same argument as in the proof of Theorem \ref{T2.1} we need to bound from below $\frac{m\eta_j^2}{M^{q-1}_q\delta_{j-1}}$. 
For $\delta_{j-1}$ we use the following known result from \cite{VT156} (see Theorem 7.3 there).

\begin{Theorem}\label{T3.2} Let $1<q\le 2$. Then
$$
\e_k(\Tr( Q_n)_q,L_\infty) \ll  \left\{\begin{array}{ll} n^{1/2}(| Q_n|/k)^{1/q} (\log (4| Q_n|/k))^{1/q}, &\quad k\le 2| Q_n|,\\
n^{1/2}2^{-k/(2| Q_n|)},&\quad k\ge 2| Q_n|.\end{array} \right.
$$
\end{Theorem}

Thus, we see that, for instance, for $2^j\le 2|Q_n|$ we have
\be\label{3.3}
\frac{m\eta_j^2}{M^{q-1}_q\delta_{j-1}} \gg \frac{m 2^{(j-1)/q}}{n^22^{n(1-1/q)}n^{(d-1)(q-2+1/q)}n^{1/2}|Q_n|^{1/q}n^{1/q}}.
\ee
Therefore, choosing $m=C_1(d,q)|Q_n|^{2-1/q}n^a$ with large enough $a$ we can get from (\ref{3.3})
$$
\frac{m\eta_j^2}{M^{q-1}_q\delta_{j-1}} \ge C(d,q)2^j.
$$
This leads us to the following Theorem \ref{T3.3} in the case $1<q\le 2$. We are limited to the case $1<q\le 2$ because Theorem \ref{T3.2} is proved in \cite{VT156} for $1<q\le 2$. In the case $2<q<\infty$ it is not difficult to derive an analog of Theorem \ref{T3.2} with the factor $n^{1/2}$ replaced by $n^{c(d)}$ with some $c(d)$. This gives Theorem \ref{T3.3} in the case $2<q<\infty$.

\begin{Theorem}\label{T3.3} Let $1<q<\infty$. There are numbers $C_1(d,q)$ and $a(d,q)$ such that for any $n\in \N$ there   exists a set of $m \le C_1(d,q)|Q_n|^{2-1/q}n^{a(d,q)}$ points $\xi^j\in \T^d$, $j=1,\dots,m$ such that for any $f\in \Tr(Q_n)$ 
we have
$$
C_2\|f\|_q^q \le \frac{1}{m}\sum_{j=1}^m |f(\xi^j)|^q \le C_3\|f\|_q^q.
$$
\end{Theorem}

\section{The Marcinkiewicz-type theorem in $L_2$. Number theoretical technique}
\label{L2}

In this section we use the technique developed in \cite{VT27}. This technique is based on elementary number theoretical constructions, which go back to 
the construction of the Korobov cubature formulas (see, for instance, \cite{TBook}). On one hand this technique gives weaker results than the probabilistic technique developed in Sections \ref{L1} and \ref{Lq}. It is limited to 
$L_2$ and produces, for instance, for $\Tr(Q_n)$ a good set of the size $\le2d |Q_n|^2$. 
On the other hand this technique gives stronger results than the probabilistic technique. It  gives the equality between the $L_2$ norm and the discrete $\ell_2$ norm of a polynomial from $\Tr(Q)$. 

\begin{Lemma}\label{iL3.1} Let $p$  be a prime 
and $\L$ be a finite subset of $\Z^d$ such that
\be\label{i3.1}
|\L| < (p-1)/d .
\ee
Then there is a natural number $a\in I_p := [1,p)$ such that for all
$\mathbf m\in\L$, $\mathbf m\ne\mathbf 0$
\be\label{i3.2}
m_1 + am_2 +\dots+a^{d-1}m_d\not\equiv 0\qquad \pmod {p}.
\ee
\end{Lemma}

\begin{proof} Let $a\in I_p$ be a natural number. We consider the congruence
\be\label{i3.3}
m_1 + am_2 +\dots+a^{d-1} m_d \equiv 0 \qquad\pmod{p} .
\ee
For the fixed vector $\mathbf m =(m_1,\dots,m_d )$ we denote by
$A_p (\mathbf m)$ the set of natural numbers $a\in I_p$ which are
solutions of the congruence (\ref{i3.3}). It is well-known that for
$\mathbf m\ne\mathbf 0$, $|m_j | < p$, $j = 1,\dots,d$ the number
$\bigl|A_p (\mathbf m)\bigr|$ of the elements of the set $A_p (\mathbf m)$
satisfies the inequality
\be\label{i3.4}
\bigl|A_p (\mathbf m)\bigr| \le d - 1 < d.
\ee
We denote by $G$ the set of the numbers $a$ for which there is a nontrivial 
solution $\mathbf m\in\L$ of the congruence (\ref{i3.3}), that is
$$
G=\cup_{\mathbf m\in\L,\bm \neq \mathbf 0}A_p (\mathbf m) .
$$
Let us estimate the number $|G|$ of elements of the set $G$. By
(\ref{i3.4}) and (\ref{i3.1}) we have
\be\label{i3.5}
|G|\le\sum_{\mathbf m\in\L,\bm \neq \mathbf 0}\bigl|A_p (\mathbf m)\bigr| <
d\bigl|\L\bigr| < p-1.
\ee
This means that there exists a number $a\in I_p$ which does not belong to
the set $G$. This $a$ is the required number by the
definition of the set $G$. The lemma is proved.
\end{proof}

We will need the following simple lemma. 

\begin{Lemma}\label{L4.2} Let $p$, $\L$, and $a$ be from Lemma \ref{iL3.1}. Consider the set $\{\xi^\nu\}_{\nu=1}^p$ of points $\xi^\nu = (\xi^\nu_1,\dots,\xi^\nu_d)$, $\xi^\nu_j:=\{\nu a^{j-1}/p\}2\pi$. Here $\{x\}$ denotes the fractional part of $x$. Then for any polynomial $t\in \Tr(\L)$ 
$$
\frac{1}{p} \sum_{\nu=1}^p t(\xi^\nu) = (2\pi)^{-d} \int_{\T^d} t(\bx)d\bx.
$$
\end{Lemma}
\begin{proof} It is sufficient to prove this lemma for $t(\bx) =e^{i(\bm,\bx)}$, $\bm\in \L$. We have
$$
\frac{1}{p} \sum_{\nu=1}^p e^{i(\bm,\xi^\nu)}= \frac{1}{p} \sum_{\nu=1}^p e^{i2\pi(m_1+m_2a+\dots+m_da^{d-1})\nu/p}.
$$
By the choice of $a$ in Lemma \ref{iL3.1} we obtain that the above sum is equal to $0$ if $\bm\in\Lambda$, $\bm \neq \mathbf 0$ and, obviously, is equal to $1$ if  $\bm = \mathbf 0$. 

\end{proof}

Let $Q$ be a finite subset of $Z^d$. We associate with it the following set
$$
\L :=\L(Q):= \{\bm-\bk: \bm\in Q,\bk\in Q\}.
$$
Then $|\L| \le |Q|^2$.
For a given $Q\subset \Z^d$ let $p$ be a smallest prime satisfying  $|\L(Q)| <(p-1)/d$ and let $a\in I_p$ be a number from Lemma \ref{iL3.1}. Define the set $\{\xi^\nu\}_{\nu=1}^p$ of points $\xi^\nu = (\xi^\nu_1,\dots,\xi^\nu_d)$, $\xi^\nu_j:=\{\nu a^{j-1}/p\}2\pi$. Here $\{x\}$ denotes the fractional part of $x$. Consider the following operator $T:=T^p_Q$, which maps a vector $b=(b_1,\dots,b_p)$ to a polynomial $T(b) \in \Tr(Q)$, by the rule
$$
T(b)(\bx) = \frac{1}{p} \sum_{\nu=1}^p b_\nu \D_Q(\bx-\xi^\nu),
$$
where
$$
\D_Q(\bx):= \sum_{\bk\in Q}e^{i(\bk,\bx)}.
$$

\begin{Proposition}\label{P4.1} The operator $T$ has the following two properties.

{\bf A.} For any $t\in \Tr(Q)$ we have
$$
T(t,\bx):= \frac{1}{p} \sum_{\nu=1}^p t(\xi^\nu) \D_Q(\bx-\xi^\nu) = t(\bx).
$$

{\bf B.} The following inequality holds
$$
\|T(b)\|_2 \le \left(\frac{1}{p}\sum_{\nu=1}^p |b_\nu|^2\right)^{1/2} = \|bp^{-1/2}\|_{\ell_2}.
$$
\end{Proposition}
\begin{proof} We first prove property {\bf A}. Clearly, it is sufficient to prove it for 
$t(\bx) = e^{i(\bm,\bx)}$, $\bm\in Q$. We have
$$
\frac{1}{p} \sum_{\nu=1}^p e^{i(\bm,\xi^\nu)} \D_Q(\bx-\xi^\nu) = \frac{1}{p} \sum_{\nu=1}^p e^{i(\bm,\xi^\nu)} \sum_{\bk\in Q}e^{i(\bk,\bx)-i(\bk,\xi^\nu)}
$$
\be\label{4.6}
=  \sum_{\bk\in Q}e^{i(\bk,\bx)}\frac{1}{p} \sum_{\nu=1}^p e^{i(\bm-\bk,\xi^\nu)}.
\ee
The fact $\bm\in Q$, $\bk\in Q$ implies that $\bm-\bk\in\L$. Applying Lemma \ref{L4.2}, we complete the proof of property {\bf A}. 

Second, we prove property {\bf B}. We have
$$
\|T(b)\|_2^2 = \frac{1}{p^2}\frac{1}{(2\pi)^d}\int_{\T^d} \left|\sum_{\nu=1}^p b_\nu \D_Q(\bx-\xi^\nu)\right|^2 d\bx
$$
$$
 = \frac{1}{p^2}\frac{1}{(2\pi)^d}\int_{\T^d} \sum_{1\le\nu,\nu'\le p} b_\nu\D_Q(\bx-\xi^\nu){\overline \D_Q(\bx-\xi^{\nu'})} {\bar b}_{\nu'} d\bx
 $$
 \be\label{4.7}
 =  \frac{1}{p^2}\sum_{1\le\nu,\nu'\le p} b_\nu\D_Q(\xi^{\nu'}-\xi^\nu){\bar b}_{\nu'}.
 \ee
In the right hand side of (\ref{4.7}) we have a quadratic form of the vector $bp^{-1/2}$ with the matrix $\D$ with elements
$$
\D_{\nu',\nu} = \frac{1}{p}\D_Q(\xi^{\nu'}-\xi^\nu).
$$
Let $\lambda (\D)$ be the largest in absolute value eigenvalue of the matrix $\D$. Then we can continue (\ref{4.7}) by
\be\label{4.8}
\le |\lambda(\D)| \|bp^{-1/2}\|_{\ell_2}^2.
\ee
Let us prove that 
\be\label{4.9}
|\lambda(\D)| \le 1.
\ee
Consider the matrix $\D^2$. Its $(u,v)$ element is given by 
\be\label{4.10} 
(\D^2)_{u,v} = \frac{1}{p^2} \sum_{n=1}^p \D_Q(\xi^{u}-\xi^n)\D_Q(\xi^n-\xi^v).
\ee
Function $\D_Q(\xi^{u}-\bx)\D_Q(\bx-\xi^v)$ belongs to $\Tr(\L)$ and, therefore, by Lemma \ref{L4.2} we get
$$
(\D^2)_{u,v} =\frac{1}{p} \frac{1}{(2\pi)^d} \int_{\T^d} \D_Q(\xi^{u}-\bx)\D_Q(\bx-\xi^v)d\bx
$$
$$
= \frac{1}{p} \D_Q(\xi^{u}-\xi^v) = \D_{u,v}.
$$
Thus, $\D^2 = \D$ and, therefore, (\ref{4.9}) holds. Combining relations 
(\ref{4.7}) -- (\ref{4.9}) we complete the proof of property {\bf B} and Proposition \ref{P4.1}.
\end{proof}

\begin{Theorem}\label{T4.1} Let $Q$ be a finite subset of $\Z^d$ and let $\L(Q)$, $p$, $a$, and the set $\{\xi^\nu\}_{\nu=1}^p$ be associated with $Q$ as above. Then for any $t\in \Tr(Q)$ we have
$$
\|t\|_2^2 = \frac{1}{p}\sum_{\nu=1}^p |t(\xi^\nu)|^2.
$$
\end{Theorem}
\begin{proof}  By property {\bf A} of Proposition \ref{P4.1} we get
\be\label{4.11}
t(\bx)= \frac{1}{p} \sum_{\nu=1}^p t(\xi^\nu) \D_Q(\bx-\xi^\nu).
\ee
By property {\bf B} of Proposition \ref{P4.1} we obtain from here
\be\label{4.12}
\|t\|_2^2 \le \frac{1}{p}\sum_{\nu=1}^p |t(\xi^\nu)|^2.
\ee

We now prove the inequality opposite to (\ref{4.12}). 
We have 
$$
\frac{1}{p}\sum_{\nu=1}^p |t(\xi^\nu)|^2= \frac{1}{p}\sum_{\nu=1}^p t(\xi^\nu){\bar t}(\xi^\nu) = \frac{1}{p}\sum_{\nu=1}^p {\bar t}(\xi^\nu) \frac{1}{(2\pi)^d}\int_{\T^d} t(\bx)\D_Q(\xi^\nu-\bx)d\bx
$$
$$
= \frac{1}{(2\pi)^d}\int_{\T^d} t(\bx) \frac{1}{p}\sum_{\nu=1}^p {\bar t}(\xi^\nu)\D_Q(\xi^\nu-\bx)d\bx \le \|t\|_2 \left\|\frac{1}{p}\sum_{\nu=1}^p {\bar t}(\xi^\nu)\D_Q(\xi^\nu-\bx)\right\|_2 .
$$
By property {\bf B} of Proposition \ref{P4.1} we continue
$$
\le \|t\|_2 \left(\frac{1}{p}\sum_{\nu=1}^p |t(\xi^\nu)|^2\right)^{1/2}.
$$
This implies
\be\label{4.13}
\left(\frac{1}{p}\sum_{\nu=1}^p |t(\xi^\nu)|^2\right)^{1/2} \le \|t\|_2.
\ee

\end{proof}

{\bf Proof of Theorem \ref{NOUth}.} We will prove a somewhat more general statement here. We begin with the following lemma from \cite{NOU} (see Lemma 2 there). 

\begin{Lemma}\label{NOUL2} Let a system of vectors $\bv_1,\dots,\bv_M$ from $\bbC^N$ have the following properties: for all $\bw\in \bbC^N$
\be\label{4.14}
\sum_{j=1}^M |\<\bw,\bv_j\>|^2 = \|\bw\|_2^2
\ee
and
\be\label{4.15}
\|\bv_j\|_2^2 = N/M,\qquad j=1,\dots,M.
\ee
Then there is a subset $J\subset \{1,2,\dots,M\}$ such that for all $\bw\in\bbC^N$
\be\label{4.16}
c_0 \|\bw\|_2^2 \le \frac{M}{N} \sum_{j\in J} |\<\bw,\bv_j\>|^2 \le C_0\|\bw\|_2^2,
\ee
where $c_0$ and $C_0$ are some absolute positive constants. 
\end{Lemma} 

\begin{Remark}\label{R4.1} For the cardinality of the subset $J$ from Lemma \ref{NOUL2} we have
$$
c_0 N \le |J| \le C_0 N.
$$
\end{Remark}
\begin{proof} We write the inequalities (\ref{4.16}) for $\bv_k$, $k=1,\dots,M$
\be\label{4.17}
c_0 \|\bv_k\|_2^2 \le \frac{M}{N} \sum_{j\in J} |\<\bv_k,\bv_j\>|^2 \le C_0\|\bv_k\|_2^2,
\ee
Summing up the inequalities (\ref{4.17}) over $k$ from $1$ to $M$ and taking into account (\ref{4.15}) and (\ref{4.14}) we obtain
\be\label{4.18} 
c_0 N \le \frac{M}{N}\sum_{j\in J} \|\bv_j\|_2^2 \le C_0 N.
\ee
Using (\ref{4.15}) again, we complete the proof.
\end{proof}
We use Lemma \ref{NOUL2} to prove the following result.
\begin{Theorem}\label{T4.2} Let  $\Omega_M=\{x^j\}_{j=1}^M$ be a discrete set with the probability measure $\mu(x^j)=1/M$, $j=1,\dots,M$. Assume that 
$\{u_i(x)\}_{i=1}^N$ is an orthonormal on $\Omega_M$ system (real or complex). Assume in addition that this system has the following property: for all $j=1,\dots, M$ we have
\be\label{4.19}
\sum_{i=1}^N |u_i(x^j)|^2 = N.
\ee
Then there is an absolute constant $C_1$ such that there exists a subset $J\subset \{1,2,\dots,M\}$ with the property:  $m:=|J| \le C_1 N$ and
 for any $f\in X_N:= \sp\{u_1,\dots,u_N\}$ we have  
$$
C_2 \|f\|_2^2 \le \frac{1}{m}\sum_{j\in J} |f(x^j)|^2 \le C_3 \|f\|_2^2, 
$$
where $C_2$ and $C_3$ are absolute positive constants. 
\end{Theorem}
\begin{proof} Define the column vectors
\be\label{4.20} 
\bv_j := M^{-1/2}(u_1(x^j),\dots,u_N(x^j))^T,\qquad j=1,\dots, M.
\ee
Then our assumption (\ref{4.19}) implies that the system $\bv_1,\dots,\bv_M$ satisfies (\ref{4.15}). For any $\bw=(w_1,\dots,w_N)^T\in \bbC^N$ we have 
$$
\sum_{j=1}^M |\<\bw,\bv_j\>|^2 = \frac{1}{M}\sum_{j=1}^M \sum_{i,k=1}^N w_i {\bar w}_k u_i(x^j){\bar u}_k(x^j) = \sum_{i=1}^N |w_i|^2
$$
by the orthonormality assumption. This implies that the system $\bv_1,\dots,\bv_M$ satisfies (\ref{4.14}). 

Applying Lemma \ref{NOUL2} and Remark \ref{R4.1} we complete the proof of 
Theorem \ref{T4.2}.

\end{proof}

We now complete the proof of Theorem \ref{NOUth}. Let $Q\subset \Z^d$ be a finite subset. Then there is an $\bN$ such that $Q\subset \Pi(\bN)$. Define $M:= \vartheta(\bN)$ and
$\Omega_M := \{\bx^\mathbf n\}_{\mathbf n\in P(\bN)}$ (see the Introduction). Consider the system $u_\bk(\bx) := e^{i(\bk,\bx)}$, $\bk\in Q$, defined on $\Omega_M$. 
It is well known that this system is orthonormal on $\Omega_M$. The property (\ref{4.19}) is obvious. 
Applying Theorem \ref{T4.2} and taking into account (\ref{b1.5}) we complete 
the proof of Theorem \ref{NOUth}.
 
\section{The entropy numbers of $\Tr(Q_n)_1$}
\label{entropy}

Proof of Theorem \ref{T2.1} is based on the known result on the behavior of the entropy numbers \newline $\e_k(\Tr(Q_n)_1,L_\infty)$, which is proved in \cite{VT156} and formulated in Theorem \ref{T2.2}. The proof of Theorem \ref{T2.2} from \cite{VT156} is rather technically involved -- it is based on the Riesz product technique for the hyperbolic cross polynomials. There is no analog of bivariate Riesz product technique in the case $d\ge 3$. In this section we present a technique, which works for all $d$. This technique gives a weaker result in case $d=2$ than in Theorem \ref{T2.2}. Instead of the extra factor $n^{1/2}$ in Theorem \ref{T2.2} this simpler technique gives an extra factor 
 $n$. It is a simplified version of the technique developed in \cite{VT156}. 

Following \cite{VT156}, we present here a construction of an orthonormal basis, which is based on the wavelet theory. This construction is taken from \cite{OO}. Let $\de$ be a fixed number, $0<\de\le1/3$, and let $\hat \ff(\la) = \hat \ff_\de(\la)$, $\la\in \R$, be a sufficiently smooth function (for simplicity, real-valued and even) equal $1$ for $|\la| \le (1-\de)/2$, equal to $0$ for $|\la|>(1+\de)/2$ and such that the integral translates of its square constitute a partition of unity:
\be\label{O1}
\sum_{k\in\Z}(\hat\ff(\la+k))^2 =1,\qquad \la\in\R.
\ee
It is known that condition (\ref{O1}) is equivalent to the following property: The set of 
functions $\Phi:=\{\ff(\cdot+l)\}_{l\in\Z}$, where
$$
\ff(x) = \int_\R\hat\ff(\la)e^{2\pi i\la x}d\la,
$$
is an orthonormal system on $\R$:
\be\label{O1'}
\int_\R\ff(x+k)\ff(x+l)dx =\de_{k,l},\qquad k,l\in\Z.
\ee
Following \cite{OO} define
$$
\theta(\la) := \left(((\hat\ff(\la/2))^2-(\hat\ff(\la))^2\right)^{1/2}
$$
and consider, for $n=0,1,\dots$, the trigonometric polynomials
\be\label{O3}
\Psi_n(x):=2^{-n/2}\sum_{k\in\Z}\theta(k2^{-n})e^{2\pi ikx}.
\ee
Introduce also the following dyadic translates of $\Psi_n$:
$$
\Psi_{n,j}(x):= \Psi_n(x-(j+1/2)2^{-n}),
$$
and define the sequence of polynomials $\{T_k\}_{k=0}^\infty$
\be\label{O4}
T_0(x):=1,\qquad T_k(x):= \Psi_{n,j}(x)  
\ee
if $k=2^n+j$, $n=0,1,\dots$, $0\le j<2^n$. Note that $T_k$ is the trigonometric polynomial such that
\be\label{O4'}
\hat T_k(\nu) =0\quad\text{if} \quad |\nu|\ge 2^n(1+\de)\quad \text{or}\quad 
|\nu|\le 2^{n-1}(1-\de).
\ee

It is proved in \cite{OO} that the system $\{T_k\}_{k=0}^\infty$ is a complete orthonormal basis in all $L_p$, $1\le p\le \infty$ (here, $L_\infty$ stands for the space of continuous functions) of $1$-periodic functions. Also, it is proved in \cite{OO} that
\be\label{O12}
|\Psi_n(x)| \le C(\kappa,\de)2^{n/2}(2^n|\sin \pi x|+1)^{-\kappa}
\ee
with $\kappa$ determined by the smoothness of $\hat\ff(\la)$. In particular, we can always make $\kappa >1$ assuming that $\hat\ff(\la)$ is smooth enough. 
It is more convenient for us to consider $2\pi$-periodic functions. We define $\V:=\{v_k\}_{k=0}^\infty$ with 
$v_k(x):=T_k(x/(2\pi))$ for $x\in [0,2\pi)$. 

In the multivariate case of $\bx=(x_1,\dots,x_d)$ we define the system $\V^d$
as the tensor product of the univariate systems $\V$. Namely, $\V^d :=\{v_\bk(\bx)\}_{\bk\in\Z_+^d}$, where
$$
v_\bk(\bx):= \prod_{i=1}^d v_{k_i}(x_i), \quad \bk=(k_1,\dots,k_d).
$$
Denote
 $$
 \rho^+(\bs):=\{ \bk =(k_1,\dots,k_d)\in \rho(\bs) : k_i\ge 0, i=1,2,\dots\}.
 $$
 Property (\ref{O12}) implies the following simple lemma.
 \begin{Lemma}\label{LO1} We have
 $$
 \|\sum_{\bk\in\rho^+(\bs)} a_\bk v_\bk\|_\infty \le C(d,\kappa,\de)2^{\|\bs\|_1/2} \max_\bk|a_\bk|.
 $$
 \end{Lemma}
 
 We use the notation
 $$
 f_\bk:= \<f,v_\bk\> = (2\pi)^{-d} \int_{\T^d} f(\bx)v_\bk(\bx)d\bx.
 $$
 Denote
 $$
 Q_n^+ := \{ \bk =(k_1,\dots,k_d)\in Q_n : k_i\ge 0, i=1,2,\dots\},\qquad \theta_n:=\{\bs: \|\bs\|_1=n\},
 $$
 $$
 \V(Q_n) := \{f: f=\sum_{\bk\in Q_n^+} c_\bk v_\bk\},\quad \V(Q_n)_A := \{f\in\V(Q_n): \sum_{\bk\in Q_n^+} |f_\bk| \le 1\}. 
 $$
The following theorem was proved in \cite{VT156}.

\begin{Theorem}\label{T6.1} Let $d=2$. For any $f\in \V(Q_n)$ we have
 $$
 \sum_{\bk\in Q_n^+} |f_\bk|  \le C|Q_n|^{1/2}\|f\|_1,
 $$
 where the constant $C$ may depend on the choice of $\hat\ff$.
 \end{Theorem} 
 
 We prove here an analog of Theorem \ref{T6.1} which covers all $d$ but gives a weaker inequality than Theorem \ref{T6.1} for $d=2$. 
 
 \begin{Theorem}\label{T6.2}  For any $f\in \V(Q_n)$ we have
 $$
 \sum_{\bk\in Q_n^+} |f_\bk|  \le C(d)n^{(d-1)/2}|Q_n|^{1/2}\|f\|_1,
 $$
 where the constant $C(d)$ may depend on $d$ and the choice of $\hat\ff$.
 \end{Theorem} 
 \begin{proof} Denote
 $$
 t_\bs := \sum_{\bk\in\rho^+(\bs)}(\sign f_\bk) v_\bk.
 $$
By Lemma \ref{LO1}
$$
\|t_\bs\|_\infty \ll 2^{\|\bs\|_1/2}.
$$
Then we get
$$
\sum_{\|\bs\|_1\le n} \sum_{\bk\in\rho^+(\bs)} |f_\bk|  = \<f,\sum_{\|\bs\|_1\le n} t_\bs\>
$$
$$
\le \sum_{\|\bs\|_1\le n} \|f\|_1 \|t_\bs\|_\infty \ll \sum_{\|\bs\|_1\le n} \|f\|_1 2^{\|\bs\|_1/2} \ll n^{d-1}2^n \|f\|_1.
$$
\end{proof}

Further, we use the technique developed in \cite{VT156}, which is based on the following two steps strategy. At the first step we obtain bounds of the best $m$-term approximations with respect to a dictionary. At the second step we use general inequalities relating the entropy numbers to the best $m$-term approximations.  
Let $\D=\{g_j\}_{j=1}^N$ be a system of elements of cardinality $|\D|=N$ in a Banach space $X$. Consider best $m$-term approximations of $f$ with respect to $\D$
$$
\sigma_m(f,\D)_X:= \inf_{\{c_j\};\Lambda:|\Lambda|=m}\|f-\sum_{j\in \Lambda}c_jg_j\|.
$$
For a function class $F$ set
$$
\sigma_m(F,\D)_X:=\sup_{f\in F}\sigma_m(f,\D)_X.
$$

We now need the following lemma from \cite{VT156}.

\begin{Lemma}\label{L6.2} Let $2\le p<\infty$. Let $\V_n^1 := \{v_\bk:\bk \in Q_n\}$. Then
\be\label{4.9v}
\sigma_m(\V(Q_n)_A,\V_n^1)_p \ll |Q_n|^{1/2-1/p}m^{1/p-1}.
\ee
\end{Lemma}

Lemma \ref{L6.2} and Theorem \ref{T6.1} imply.
\begin{Lemma}\label{L6.3} Let $2\le p<\infty$. Let $\V_n^1 := \{v_\bk:\bk \in Q_n\}$. Then
$$
\sigma_m(\V(Q_n)_1,\V_n^1)_p \ll n^{(d-1)/2}(|Q_n|/m)^{1-1/p}.
$$
\end{Lemma}

Proceeding as in \cite{VT156} we obtain from here the following weaker analog
of Theorem \ref{T2.2} (see Theorem 7.4 in \cite{VT156}).

\begin{Theorem}\label{T6.3} We have 
$$
\e_k(\Tr( Q_n)_1,L_\infty) \ll  \left\{\begin{array}{ll} n^{d/2}(| Q_n|/k) \log (4| Q_n|/k), &\quad k\le 2| Q_n|,\\
n^{d/2}2^{-k/(2| Q_n|)},&\quad k\ge 2| Q_n|.\end{array} \right.
$$
\end{Theorem}

{\bf Acknowledgements.}
The author is grateful to Boris Kashin for very useful comments.


\begin{thebibliography}{9999}

\bibitem{BSS} J. Batson, D.A. Spielman, and N. Srivastava (2012), 
Twice-Ramanujan Sparsifiers, SIAM J. Comput., {\bf 41} (2012), 1704--1721.

 
\bibitem{BLM} J. Bourgain, J. Lindenstrauss and V. Milman, Approximation of zonoids by zonotopes, Acta Math., {\bf 162} (1989), 73--141.

\bibitem{DTU} Ding Dung, V.N. Temlyakov, and T. Ullrich, Hyperbolic Cross Approximation, arXiv:1601.03978v2 [math.NA] 2 Dec 2016.

\bibitem{GZ} E. Gine and J. Zinn, Some limit theorems for empirical processes, 
Ann. Prob., {\bf 12} (1984), 929--989.

\bibitem{Ka} B.S. Kashin, Lunin's method for selecting large submatrices with small norm, Matem. Sb., {\bf 206} (2015), 95--102. 
 
\bibitem{KT3} B.S. Kashin and V.N. Temlyakov, On a norm and related applications, Mat. Zametki {\bf 64} (1998),  637--640.
 
\bibitem{KT4} B.S. Kashin and V.N. Temlyakov, On a norm and approximation characteristics of classes of functions of several variables,
Metric theory of functions and related problems in analysis, Izd. Nauchno-Issled. Aktuarno-Finans. Tsentra (AFTs), Moscow, 1999, 69--99.

\bibitem{KaTe03} B.S. Kashin and V.N. Temlyakov, The volume estimates and their applications,  East J. Approx., {\bf 9}  (2003), 469--485.

\bibitem{KoTe} S.V. Konyagin and V.N. Temlyakov, The Entropy in Learning Theory. Error Estimates, Constr. Approx., {\bf 25} (2007), 1--27.

\bibitem{MSS} A. Marcus, D.A. Spielman, and N. Srivastava,
Interlacing families II: Mixed characteristic polynomials and the Kadison-Singer problem, Annals of Math., {\bf 182} (2015), 327--350.

\bibitem{NOU} S. Nitzan, A. Olevskii, and A. Ulanovskii,
{Exponential frames on unbounded sets}, Proc. Amer. Math. Soc., {\bf 144} (2016),109--118.

\bibitem{OO} D. Offin and K. Oskolkov, A note on orthonormal polynomial bases and wavelets, Constructive Approx. {\bf 9} (1993), 319--325.

\bibitem{Rud} M. Rudelson, Almost orthogonal submatrices of an orthogonal matrix, Izrael J. Math., {\bf 111} (1999), 143--155.

\bibitem{Ta} M. Talagrand, The generic chaining, Berlin: Springer, 2005. 

\bibitem{VT27} V.N. Temlyakov, On reconstruction of multivariate periodic functions based on their values at the knots of number-theoretical nets, Analysis Mathematica, {\bf12} (1986), 287--305. 


\bibitem{Tmon} V.N. Temlyakov, Approximation of functions with bounded mixed derivative, Trudy MIAN, {\bf 178} (1986), 1--112. English transl. in Proc. Steklov Inst. Math., {\bf 1} (1989).

\bibitem{TBook}  V.N. Temlyakov, {\em Approximation of periodic functions}, 
Nova Science Publishes, Inc., New York.,  1993.

\bibitem{Tbook} V.N. Temlyakov, Greedy approximation, Cambridge University
Press, 2011.

\bibitem{VT156} V.N. Temlyakov, On the entropy numbers of the mixed smoothness function classes, arXiv:1602.08712v1 [math.NA] 28 Feb 2016.

\bibitem{VT161} V.N. Temlyakov,
The Marcinkiewicz-type discretization theorems, arXiv:1703.03743v1 [math.NA], 10 Mar 2017.

\bibitem{Z} A. Zygmund, Trigonometric Series, Cambridge University Press, 1959.
 
\end{thebibliography}
\end{document}